\documentclass[12pt]{article}
\usepackage{amsfonts}
\usepackage{amsmath,amsthm}
\usepackage{amssymb}
\usepackage{eucal,color,graphicx}
\usepackage[bookmarksnumbered, colorlinks, plainpages]{hyperref}
\newtheorem{Theorem}{Theorem}
\newtheorem{Lemma}[Theorem]{Lemma}
\newtheorem{Problem}[Theorem]{Problem}

\begin{document}

\baselineskip 17pt
\title{\textbf{On a problem from the Kourovka Notebook}\thanks{The author is supported by an NNSF of China (grant No. 11371335), the Start-up Scientific Research Foundation of Nanjing Normal University (grant No. 2015101XGQ0105) and a project funded by the Priority Academic Program Development of Jiangsu Higher Education Institutions.}}

\author{{Xiaoyu Chen}\\
{\small School of Mathematical Sciences and Institute of Mathematics, Nanjing Normal University,}\\
{\small Nanjing 210023, P. R. China}\\
{\small E-mail: jelly@njnu.edu.cn}}
\date{}
\maketitle

\begin{abstract}
In this manuscript, a solution to Problem 18.91(b) in the Kourovka Notebook is given by proving the following theorem. Let $P$ be a Sylow $p$-subgroup of a group $G$ with $|P| = p^n$. Suppose that there is an integer $k$ such that $1 < k < n$ and every subgroup of $P$ of
order $p^k$ is $S$-propermutable in $G$, and also, in the case that $p=2$, $k = 1$ and $P$ is non-abelian, every cyclic subgroup
of $P$ of order $4$ is $S$-propermutable in $G$. Then $G$ is $p$-nilpotent.\par
\end{abstract}
\renewcommand{\thefootnote}{\empty}
\footnotetext{Keywords: Finite group, Propermutable subgroups, $S$-propermutable subgroups, $p$-nilpotence.}
\footnotetext{Mathematics Subject Classification (2010): 20D10, 20D20.}

Recall that a subgroup $H$ of a group $G$ is said to be \textit{propermutable} (resp. \textit{$S$-propermutable}) \cite{Yi} in $G$, if there is a subgroup $B$ of $G$ such that $G = N_G(H)B$ and $H$ permutes with every subgroup (resp. Sylow subgroup) of $B$. The aim of this manuscript is to give a solution to the following problem proposed by A. N. Skiba in the Kourovka Notebook.

\begin{Problem} $($see Problem $18.91(b)$ in \textup{\cite{Maz}}.$)$ Let $P$ be a non-abelian Sylow $2$-subgroup of a group $G$ with $|P| = 2^n$.
Suppose that there is an integer $k$ such that $1 < k < n$ and every subgroup of $P$ of
order $2^k$ is propermutable in $G$, and also, in the case of $k = 1$, every cyclic subgroup
of $P$ of order $4$ is propermutable in $G$. Is it true that then $G$ is $2$-nilpotent?\par\end{Problem}

%18.91. (a) Is there a finite group $G$ with subgroups $A\leq B\leq G$ such that $A$ is propermutable in $G$ but $A$ is not propermutable in $B$?\par

A subgroup $H$ of a group $G$ is said to satisfy \textit{$\Pi$-property} \cite{Li} in $G$ if for every chief factor $L/K$ of $G$, $|G/K:N_{G/K}(HK/K\cap L/K)|$ is a $\pi(HK/K\cap L/K)$-number. Now we can establish the relationship between $S$-propermutable subgroups and subgroups which satisfy $\Pi$-property.\par

\begin{Lemma} If a $p$-subgroup $H$ is $S$-propermutable in a group $G$, then $H$ satisfies $\Pi$-property in $G$.
\end{Lemma}

\begin{proof} In view of \cite[Lemma 2.3(1)]{Yi} and \cite[Proposition 2.1(1)]{Li}, we only need to prove that $|G:N_G(H\cap N)|$ is a $p$-number for any minimal normal subgroup $N$ of $G$ by induction. If $N$ is abelian, then $|G:N_G(H\cap N)|$ is a $p$-number by \cite[Lemma 2.3(4)]{Yi}. We may, therefore, assume that $N$ is non-abelian.\par
Now we shall prove that $H\cap N=1$. As $H$ is $S$-propermutable in $G$, $G$ has a subgroup $B$ such that $G=N_G(H)B$  and $H$ permutes with every Sylow subgroup of $B$. Clearly, $H^G\leq HB$. If $H^G\cap N=1$, then $H\cap N=1$. Hence we may assume that $N\leq H^G\leq HB$. Then for any Sylow $q$-subgroup $N_q$ of $N$ with $q\neq p$, $B$ has a Sylow $q$-subgroup $B_q$ such that $N_q=(B_q)^h\cap N$ for some $h\in H$. It follows that $H(B_q)^h\cap N=(H\cap N)((B_q)^h\cap N)=(H\cap N)N_q$, and so $H\cap N$ permutes with $N_q$. Since $N\neq (H\cap N)N_q$, $N$ has a proper normal subgroup $L$ such that either $H\cap N\leq L$ or $N_q\leq L$. Then evidently, we have that $H\cap N\leq L$. Note that for any Sylow $q$-subgroup $L_q$ of $L$, $H\cap N$ permutes with $L_q$. Repeating this argument, we can obtain that $H\cap N=1$. The lemma is thus proved.
\end{proof}

We arrive at the following main result.

\begin{Theorem}
Let $P$ be a Sylow $p$-subgroup of a group $G$ with $|P| = p^n$. Suppose that there is an integer $k$ such that $1 < k < n$ and every subgroup of $P$ of
order $p^k$ is $S$-propermutable in $G$, and also, in the case that $p=2$, $k = 1$ and $P$ is non-abelian, every cyclic subgroup
of $P$ of order $4$ is $S$-propermutable in $G$. Then $G$ is $p$-nilpotent.\end{Theorem}
\begin{proof}
In fact, by Lemma 2, this theorem follows directly from the results by using the concept of $\Pi$-property, for example, the Main Theorem in \cite{Su}.
\end{proof}

Clearly, propermutable subgroups are $S$-propermutable in $G$, and thus Problem 1 has a positive answer due to Theorem 3.

\end{document}